\documentclass[12pt,a4paper]{article}%
\usepackage[utf8]{inputenc}
\usepackage{hyperref}
\usepackage{amsmath}
\usepackage{amsfonts}
\usepackage{amssymb}
\usepackage{xcolor}
\usepackage{graphicx}%
\providecommand{\U}[1]{\protect\rule{.1in}{.1in}}
\newtheorem{theorem}{Theorem}
\newtheorem{conjecture}[theorem]{Conjecture}
\newtheorem{corollary}[theorem]{Corollary}
\newtheorem{lemma}[theorem]{Lemma}
\newtheorem{observation}[theorem]{Observation}

\newtheorem{proposition}[theorem]{Proposition}
\newenvironment{proof}[1][Proof]{\noindent\textbf{#1.} }{\ \hfill \rule{0.5em}{0.5em}\bigskip}
\graphicspath{{Slike/}}

\begin{document}

\title{The rich-neighbor chromatic index is linear in the maximum degree}
\author{Masaki Kashima$^{1}$, Mirko Petru\v{s}evski$^{2}$, Jelena Sedlar$^{3,5}$,\\Riste \v{S}krekovski$^{4,5,6}$\\[0.3cm] {\small $^{1}$ \textit{Keio University, FST, Yokohama, Japan}}\\[0.1cm] {\small $^{2}$ \textit{Ss. Cyril and Methodius University, FME,
Skopje, North Macedonia}}\\[0.1cm] {\small $^{3}$ \textit{University of Split, FGAG, Split, Croatia}}\\[0.1cm] {\small $^{4}$ \textit{University of Ljubljana, FMF, Ljubljana,
Slovenia}}\\[0.1cm] {\small $^{5}$ \textit{University of Novo Mesto, FIS, Novo Mesto,
Slovenia}}\\[0.1cm] {\small $^{6}$ \textit{Rudolfovo -- Science and Technology Centre,
Novo Mesto, Slovenia}}}
\date{}
\maketitle

\begin{abstract}
An edge of a multigraph is \emph{rich} with respect to a proper edge-coloring
if all the edges adjacent to it receive mutually distinct colors, and a proper
edge-coloring is a \emph{rich-neighbor coloring} if every non-isolated edge
has a rich edge adjacent to it. The least number of colors in such a coloring
is the rich-neighbor chromatic index $\chi_{\mathrm{rn}}^{\prime}(G)$. So far
it was studied for subcubic graphs only, and multigraphs of larger maximum
degree are considered here for the first time. The only upper bound valid in
general was the quadratic one inherited from the strong chromatic index, while
a single rich edge already requires $2\Delta-1$ colors in a $\Delta$-regular
graph. We give an upper linear bound $\chi_{\mathrm{rn}}^{\prime}%
(G)\leq6\Delta-8$ which holds for every loopless multigraph $G$ with
$\Delta=\Delta(G)\geq3,$ and we show that a coloring realizing this bound can
be found in polynomial time. Consequently $\chi_{\mathrm{rn}}^{\prime
}(G)=\Theta(\Delta)$. For $\Delta$-regular graphs with $\Delta\geq3$ the lower
bound is $2\Delta-1,$ and the upper bound is $5\Delta-8$. The bound
$5\Delta-8$ holds more generally for every multigraph with $\Delta\geq3$ which
has a dominating $P_{\geq3}$-system, and we characterize these multigraphs:
such a system exists if and only if $k_{2}(G-T)\leq2\left\vert T\right\vert $
holds for every set $T$ of vertices, where $k_{2}$ counts the components on
two vertices. In particular the bound holds for every graph with $3\leq
\Delta\leq4\delta(G)-4$. The paper is concluded with a conjecture on the tight
upper bound on $\chi_{\mathrm{rn}}^{\prime}(G).$

\end{abstract}



\section{Introduction}

All the multigraphs considered in this paper are finite and loopless, and
multiple edges are allowed. This is the setting of
\cite{PetrusevskiSkrekovski}, in which the invariant studied here was
introduced, and it is the setting in which we work throughout. A multigraph
with no multiple edges is a \emph{graph}, so all our results hold for graphs
in particular.

Let $G$ be a loopless multigraph and let $\varphi$ be a proper edge-coloring
of $G$. For a vertex $v$ we denote by $E(v)$ the set of all edges incident to
$v,$ we put $d(v)=\left\vert E(v)\right\vert $ and $\varphi(E(v))=\{\varphi
(f):f\in E(v)\}$. An edge $e$ of $G$ is \emph{rich} if all the edges adjacent
to $e$ receive mutually distinct colors. A proper edge-coloring of $G$ is a
\emph{rich-neighbor coloring} if every non-isolated edge of $G$ has at least
one rich edge adjacent to it. Here an isolated edge is an edge adjacent to no
other edge. The least number of colors in a rich-neighbor coloring of $G$ is
the \emph{rich-neighbor chromatic index} $\chi_{\mathrm{rn}}^{\prime}(G)$.
Notice that a component on one or two vertices imposes no restriction, since
in a component on two vertices every edge is rich in any proper edge-coloring.
In addition, we can color every component separately, so in what follows we
will mostly assume that our graphs are connected and of order at least three.

The invariant was introduced in \cite{PetrusevskiSkrekovski} as a relaxation
of the strong chromatic index $\chi_{s}^{\prime}(G)$. The latter is the least
number of colors in a proper edge-coloring in which \emph{every} edge is rich,
and it goes back to Fouquet and Jolivet \cite{FouquetJolivet}. A coloring in
which every edge is rich is in particular a rich-neighbor coloring, hence, as
observed in \cite{ChenNanZhou},
\[
\chi_{\mathrm{rn}}^{\prime}(G)\leq\chi_{s}^{\prime}(G)
\]
for every graph $G$. In \cite{PetrusevskiSkrekovski} it was shown that every
subcubic multigraph satisfies $\chi_{\mathrm{rn}}^{\prime}(G)\leq7$. It was
also conjectured there that $\chi_{\mathrm{rn}}^{\prime}(G)\leq5$ holds for
every connected subcubic multigraph distinct from $K_{4}$. The bound was
improved to $6$ by Chen, Nan and Zhou \cite{ChenNanZhou} for subcubic graphs
with $\mathrm{mad}(G)<36/13$ and for claw-free subcubic graphs. They in fact
obtain a coloring with the additional property that every vertex of degree two
is incident to a rich edge.

All the work quoted above concerns subcubic multigraphs, and for a larger
maximum degree the invariant is considered here for the first time. The only
bound known previously was the one inherited from the strong chromatic index,
namely $\chi_{\mathrm{rn}}^{\prime}(G)\leq\chi_{s}^{\prime}(G)\leq2\Delta
^{2}-2\Delta+1,$ which is obtained by coloring the edges greedily, since there
are at most $2\Delta^{2}-2\Delta$ edges at distance at most one from a given
edge. On the other hand, a rich edge $xy$ forces $d(x)+d(y)-1$ distinct
colors, so already a single rich edge in a $\Delta$-regular graph requires
$2\Delta-1$ colors. Thus $\chi_{\mathrm{rn}}^{\prime}$ was bounded below by a
linear and above by a quadratic function of $\Delta,$ and it was not known
which of the two is the truth. Notice that this is in sharp contrast with the
strong chromatic index, where the quadratic growth is genuine, since $\chi
_{s}^{\prime}(K_{n})=\binom{n}{2}$.

In this paper we settle that a linear upper bound does hold. Namely, we show
that every loopless multigraph $G$ with $\Delta=\Delta(G)\geq3$ satisfies
$\chi_{\mathrm{rn}}^{\prime}(G)\leq6\Delta-8,$ and that a coloring realizing
this bound can be found in polynomial time. Since a rich-neighbor coloring is
in particular a proper edge-coloring, we have $\chi_{\mathrm{rn}}^{\prime
}(G)\geq\chi^{\prime}(G)\geq\Delta,$ so this gives $\chi_{\mathrm{rn}}%
^{\prime}(G)=\Theta(\Delta)$ for every multigraph with $\Delta\geq2$. For
$\Delta$-regular graphs with $\Delta\geq3$ the lower bound is $2\Delta-1,$ and
we prove the better upper bound $5\Delta-8$. The multigraphs with $\Delta
\leq2$ are disposed of in Section 2, so that all our later statements can
assume $\Delta\geq3$.

The bound $5\Delta-8$ relies on a structure which we call a dominating
$P_{\geq3}$-system: a set of edges whose components are paths on at least
three vertices and which leaves an independent set of vertices uncovered.
Section 4 is devoted to these systems. We characterize the multigraphs which
have one, by a condition of Tutte type on the number of components on two
vertices, and we deduce that a graph has a dominating $P_{\geq3}$-system
whenever $\delta(G)\geq2$ and $\Delta\leq4\delta(G)-4$. The regular graphs are
the special case $\delta=\Delta$ of that condition. Not every multigraph has
such a system. For example, the spider obtained from $K_{1,3}$ by subdividing
every edge once has none.

\section{Preliminaries}

Let us fix the notation. For a multigraph $G$ we denote by $\Delta$ the
maximum degree $\Delta(G)$ and by $\delta(G)$ the minimum degree. For an edge
$e=uv$ we denote by $m(e)$ the number of edges of $G$ joining $u$ and $v,$ so
that $m(e)\geq1$ and
\begin{equation}
\left\vert E(u)\cup E(v)\right\vert =d(u)+d(v)-m(e). \label{For_union}%
\end{equation}

We will choose a set $R\subseteq E(G)$ of edges which we want to make rich. For
such $R$ we denote by $R(u)=E(u)\cap R$ the set of the edges of $R$ incident
to a vertex $u,$ and we put $\deg_{R}(u)=\left\vert R(u)\right\vert $.
Further, $V(R)$ is the set of all vertices incident to at least one edge of
$R$. Finally, $N_{R}(u)$ is the set of the vertices joined to $u$ by an edge
of $R$.

For an
integer $d\geq0,$ a graph $H$ is $d$\emph{-degenerate} if every subgraph of
$H$ has a vertex of degree at most $d$. Equivalently, the vertices of $H$ can
be ordered as $v_{1},\ldots,v_{n}$ so that every $v_{i}$ has at most $d$
neighbors among $v_{1},\ldots,v_{i-1}$. Such an order is obtained by
repeatedly removing a vertex of smallest degree and reversing the order of
removal, and it is called a \emph{degeneracy order} of $H$. Coloring the
vertices greedily along an order with this property uses at most $d+1$ colors,
and both the order and the coloring are found in linear time.

We start with a simple lower bound from \cite{ChenNanZhou}.

\begin{observation}
\label{Prop_lower}Let $G$ be a $\Delta$-regular graph with $\Delta\geq2$. Then
$\chi_{\mathrm{rn}}^{\prime}(G)\geq2\Delta-1$.
\end{observation}

\noindent Indeed, since $\Delta\geq2$ every edge of $G$ is adjacent to another
one, so a rich-neighbor coloring has at least one rich edge $xy.$ The sets
$\varphi(E(x))$ and $\varphi(E(y))$ have $\Delta$ elements each and meet only
in $\varphi(xy),$ which already forces $2\Delta-1$ colors. The assumption that
$G$ is a graph cannot be omitted here due to (\ref{For_union}). Neither can
the assumption of regularity be omitted, since every edge of the star
$K_{1,\Delta}$ is rich in any proper edge-coloring, which yields
$\chi_{\mathrm{rn}}^{\prime}(K_{1,\Delta})=\Delta$.

We next dispose of the multigraphs of maximum degree at most two, so that all
our later statements can be made under the assumption $\Delta\geq3$.

\begin{proposition}
\label{Prop_small}Let $G$ be a loopless multigraph with $\Delta(G)\leq2$. Then
$\chi_{\mathrm{rn}}^{\prime}(G)\leq4,$ and in fact $\chi_{\mathrm{rn}}%
^{\prime}(G)\leq3$ unless some component of $G$ is $C_{4}$ or $C_{7}$.
\end{proposition}

\begin{proof}
Every component of $G$ is an isolated vertex, a path, or a cycle $C_{n}$ with
$n\geq2$. We color the edges of every path and of every cycle $C_{n}$ with
$n\not \equiv 1$ $(\mathrm{mod}$ $3)$ by repeating the sequence $1,2,3$. If
$n\equiv1$ $(\mathrm{mod}$ $3)$ and $n\geq10$, we start with the sequence
$1,2,1,2,3,1,2$ and then repeat the sequence $1,2,3$ until all edges are
colored. For $C_{4}$ and $C_{7}$ we color the edges by $1,2,3,4$ and
$1,2,3,1,2,3,4$, respectively. It is easily verified that all these are
rich-neighbor colorings.
\end{proof}

Our approach is to prescribe in advance a set $R$ of edges which we force to
be rich, and then to color $G$ greedily so that this is achieved. To this end,
given a loopless multigraph $G$ and a set $R\subseteq E(G),$ we denote by
$\mathcal{G}$ the graph on $E(G)$ in which two distinct edges are adjacent if
they are adjacent in $G$ or they are both adjacent to an edge of $R$.

\begin{lemma}
\label{Lemma_conflict}Let $G$ be a loopless multigraph and $R\subseteq E(G)$.
Then every proper vertex-coloring of $\mathcal{G}$ is a proper edge-coloring
of $G$ in which every edge of $R$ is rich.
\end{lemma}

\begin{proof}
Adjacent edges of $G$ are adjacent in $\mathcal{G},$ so a proper
vertex-coloring $\varphi$ of $\mathcal{G}$ is a proper edge-coloring of $G$.
It remains to show that every edge $xy\in R$ is rich. Two edges $e,f\in
E(x)\cup E(y)$ are either adjacent, or they are both adjacent to $xy\in R.$
This implies they are adjacent in $\mathcal{G}$, so $\varphi(e)\not =%
\varphi(f).$
\end{proof}

By Lemma \ref{Lemma_conflict} every edge of $R$ is rich. The next lemma says
when this gives a rich-neighbor coloring of $G$. We use it in Sections 3 and 4.

\begin{lemma}
\label{Lemma_service}Let $G$ be a loopless multigraph and let $R\subseteq
E(G)$ be such that no connected component of $R$ consists of a single edge and
the set $V(G)\setminus V(R)$ is independent in $G$. If $\varphi$ is a proper
edge-coloring of $G$ in which every edge of $R$ is rich, then $\varphi$ is a
rich-neighbor coloring of $G$.
\end{lemma}

\begin{proof}
Let $e=uv$ be an edge of $G$. If $e\in R,$ then the component of $R$
containing $e$ has at least two edges, so $e$ is adjacent to another edge of
$R,$ which is rich. If $e\not \in R,$ then $u$ and $v$ are adjacent, so they
are not both in the independent set $V(G)\setminus V(R)$. Hence one of them,
say $u,$ is incident to an edge of $R,$ which is distinct from $e$ and thus a
rich edge adjacent to $e$.
\end{proof}

\section{The general bound $6\Delta-8$}

In this section we prove the general upper bound. For that purpose we take as
$R$ the edge set of a spanning tree of $G,$ and we bound the average degree of
every subgraph of $\mathcal{G}$.

\begin{lemma}
\label{Lemma_tree}Let $G$ be a connected loopless multigraph with
$\Delta=\Delta(G)\geq3$ and let $R$ be the edge set of a spanning tree of $G$.
Then $\mathcal{G}$ is $\left(  6\Delta-9\right)  $-degenerate.
\end{lemma}

\begin{proof}
Let $S\subseteq E(G)$ be non-empty, and for a vertex $u$ put $S(u)=E(u)\cap S$
and $s_{u}=\left\vert S(u)\right\vert \leq\Delta$, so $\sum_{u}s_{u}%
=2\left\vert S\right\vert $. Choose a vertex $r$ with $s_{r}>0$ and root the
spanning tree at $r$. For every vertex $u\not =r$ let $e_{u}$ be the edge of
$R$ joining $u$ to its parent in the spanning tree, and put $\varepsilon
_{u}=1$ if $e_{u}\in S$ and $\varepsilon_{u}=0$ otherwise. Notice that every
edge of $R$ is $e_{u}$ for exactly one vertex $u\not =r$.

We wish to bound the number of edges of $\mathcal{G}[S]$. Pairs $\{f,g\}$
adjacent in $G$ share an endvertex $u,$ so there are at most $\sum_{u}%
\binom{s_{u}}{2}$ of them. In the remaining edges $\{f,g\}$ of $\mathcal{G}%
[S],$ the edges $f$ and $g$ are non-adjacent and they are joined by some
$e_{u}=xu$ from $R,$ where $x$ is the parent of $u$. Then one of them lies in
$S(u)\setminus\{e_{u}\}$ and the other in $S(x)\setminus\{e_{u}\},$ so $e_{u}$
gives at most $\left(  s_{x}-\varepsilon_{u}\right)  \left(  s_{u}%
-\varepsilon_{u}\right)  \leq\left(  \Delta-1\right)  \left(  s_{u}%
-\varepsilon_{u}\right)  $ such pairs. Hence
\[
2\left\vert E(\mathcal{G}[S])\right\vert \leq s_{r}(s_{r}-1)+\sum_{u\not =%
r}\left(  s_{u}(s_{u}-1)+2\left(  \Delta-1\right)  \left(  s_{u}%
-\varepsilon_{u}\right)  \right)  .
\]
Let us show that every summand for $u\not =r$ is at most $\left(
3\Delta-4\right)  s_{u}$. Namely, if $\varepsilon_{u}=0,$ then $s_{u}%
\leq\Delta-1,$ so the summand is at most $\left(  \Delta-2\right)
s_{u}+2\left(  \Delta-1\right)  s_{u}=\left(  3\Delta-4\right)  s_{u}$. If
$\varepsilon_{u}=1,$ then from $s_{u}\leq\Delta$ we have that the summand is
at most $s_{u}\left(  \Delta-1\right)  +2\left(  \Delta-1\right)  \left(
s_{u}-1\right)  \leq\left(  3\Delta-4\right)  s_{u}$. For the root $r$ we have
$s_{r}(s_{r}-1)\leq s_{r}\left(  \Delta-1\right)  <\left(  3\Delta-4\right)
s_{r},$ since $s_{r}>0$. Thus we obtain
\[
2\left\vert E(\mathcal{G}[S])\right\vert <\left(  3\Delta-4\right)  \sum
_{u}s_{u}=\left(  6\Delta-8\right)  \left\vert S\right\vert ,
\]
so $\mathcal{G}[S]$ has a vertex of degree at most $6\Delta-9$. Thus
$\mathcal{G}$ is $\left(  6\Delta-9\right)  $-degenerate.
\end{proof}

Using Lemma \ref{Lemma_tree} we now prove the general upper bound.

\begin{theorem}
\label{Tm_main}Every loopless multigraph $G$ with $\Delta=\Delta(G)\geq3$
satisfies
\[
\Delta\leq\chi_{\mathrm{rn}}^{\prime}(G)\leq6\Delta-8.
\]
Moreover, a rich-neighbor $\left(  6\Delta-8\right)  $-coloring can be found
in polynomial time.
\end{theorem}

\begin{proof}
The lower bound holds since a rich-neighbor coloring is a proper
edge-coloring, so $\chi_{\mathrm{rn}}^{\prime}(G)\geq\chi^{\prime}%
(G)\geq\Delta$. For the upper bound, notice that the components of $G$ can be
colored independently, so it suffices to color every component $H$ of $G$ with
at most $6\Delta-8$ colors. If $H$ has at most two vertices, then all its
edges are rich in any proper edge-coloring, so $\Delta(H)\leq\Delta$ colors
suffice. If $\Delta(H)\leq2,$ then four colors suffice by Proposition
\ref{Prop_small}. Hence, assume that $H$ has at least three vertices and
$\Delta(H)\geq3,$ and take for $R$ the edge set of a spanning tree of $H$. By
Lemma \ref{Lemma_tree}, $\mathcal{G}$ is $\left(  6\Delta(H)-9\right)
$-degenerate, so the greedy coloring of $\mathcal{G}$ along a degeneracy
order uses at most $6\Delta(H)-8\leq6\Delta-8$ colors. By Lemma
\ref{Lemma_conflict}, it is a proper edge-coloring of $H$ in which every edge
of $R$ is rich. The set $R$ satisfies the assumptions of Lemma
\ref{Lemma_service}, since it is connected with at least two edges and
$V(R)=V(H),$ so this is a rich-neighbor coloring of $H$.

A spanning tree, the graph $\mathcal{G}$ and a degeneracy order of
$\mathcal{G}$ are found in time polynomial in $\left\vert V(G)\right\vert
+\left\vert E(G)\right\vert ,$ so the coloring can be found in polynomial time.
\end{proof}

\noindent Thus $\chi_{\mathrm{rn}}^{\prime}(G)$ is bounded from both sides by
linear functions of $\Delta,$ implying $\chi_{\mathrm{rn}}^{\prime}%
(G)=\Theta(\Delta)$. For $\Delta$-regular graphs the lower bound can be
improved to $2\Delta-1$ by Observation \ref{Prop_lower}.

\section{Dominating $P_{\geq3}$-systems and the bound $5\Delta-8$}

In this section we show that the bound of Theorem \ref{Tm_main} can be
improved to $5\Delta-8$ whenever the multigraph has a dominating $P_{\geq3}%
$-system (defined below), and we characterize the multigraphs which have one.

We first consider sets $R$ whose components are paths. For such $R$ we obtain
a better bound than in Lemma \ref{Lemma_tree}, using a suitable order of
$E(G)$. A set $R\subseteq E(G)$ is a \emph{$P_{\geq3}$-system} if every
component of $\left(  V(R),R\right)  $ is a path on at least three vertices.
If moreover $V(R)=V(G),$ then $R$ is a \emph{$P_{\geq3}$-factor}. If
$V(G)\setminus V(R)$ is independent in $G,$ then $R$ is a \emph{dominating
$P_{\geq3}$-system}. Notice that a $P_{\geq3}$-system is a forest with no
multiple edges, and that $\deg_{R}(v)\leq2$ for every vertex $v$.

\begin{lemma}
\label{Lemma_linear}Let $G$ be a loopless multigraph with $\Delta
=\Delta(G)\geq4$ and let $R\subseteq E(G)$ be a $P_{\geq3}$-system. Then
$\mathcal{G}$ is $\left(  5\Delta-9\right)  $-degenerate.
\end{lemma}

\begin{proof}
In order to prove the claim, it is sufficient to introduce an order of
vertices of $\mathcal{G}$, i.e. edges of $E(G),$ so that each vertex of
$\mathcal{G}$ has at most $5\Delta-9$ neighbors which come after it in the
order. For that purpose, we first introduce an order in $V(G),$ so that the
paths of $R$ come first, one after another, each path $x_{0}x_{1}\cdots x_{k}$
in the order $x_{0},$ $x_{1},$ $\ldots,$ $x_{k},$ and then the vertices of
$V(G)\setminus V(R)$ in arbitrary order. We write $x\prec y$ if $x$ comes
before $y$. The \emph{tail} of an edge is its endvertex which comes first in
this order. Now, edges of $E(G)$ are ordered so that all the edges of $R$ come
before all the other edges, and within each of these two parts the edges are
ordered by their tails, the edges with the same tail being ordered
arbitrarily. Let $\nu(e)$ denote the number of $\mathcal{G}$-neighbors of an
edge $e$ which follow $e$ in this order, and we need to prove $\nu
(e)\leq5\Delta-9.$

For an edge $e=uv$ of $G,$ assume $u\prec v$ and denote $X(e)=\{u,v\}\cup
N_{R}(u)\cup N_{R}(v).$ Notice that each $\mathcal{G}$-neighbor of $e$ has one
endvertex in $X(e).$ Moreover, since $R$ is a $P_{\geq3}$-system, for $e\in R$
we have $\left\vert X(e)\right\vert \leq4$ and otherwise $\left\vert
X(e)\right\vert \leq6.$ We distinguish the following two cases.

\medskip

\noindent\textbf{Case 1:} $e=uv\in R$. In this case $X(e)$ contains $3$ or $4$
consecutive vertices of a path $P$ of $R$ containing $e.$ Notice that
$\left\vert N_{\mathcal{G}}(e)\right\vert \leq\left\vert X(e)\right\vert
\Delta-(\left\vert X(e)\right\vert -1)-1.$ Here we must subtract $(\left\vert
X(e)\right\vert -1)+1$ since edges of $P$ are counted twice in $\left\vert
X(e)\right\vert \Delta,$ while they should be counted once except for $e$
which should not be counted at all. If $\left\vert X(e)\right\vert =3,$ we
obtain $\nu(e)\leq\left\vert N_{\mathcal{G}}(e)\right\vert \leq3\Delta
-3\leq5\Delta-9$ for $\Delta\geq3.$ If $\left\vert X(e)\right\vert =4,$ then
$\left\vert N_{\mathcal{G}}(e)\right\vert \leq4\Delta-4.$ Yet, in this case
there is an edge $f\not =e$ of $P$ incident to $u$ which is counted in
$\left\vert N_{\mathcal{G}}(e)\right\vert .$ Since $f$ precedes $e,$ it is not
counted in $\nu(e),$ so $\nu(e)\leq\left\vert N_{\mathcal{G}}(e)\right\vert
-1=4\Delta-5\leq5\Delta-9$ for $\Delta\geq4.$

\medskip

\noindent\textbf{Case 2:} $e=uv\not \in R$. All $\mathcal{G}$-neighbors of $e$
are incident to a vertex of $X(e).$ Yet, if a vertex $x\in X(e)$ precedes $u$
in the vertex order, then all edges of $G$ incident to $x$ precede $e$ and are
not counted in $\nu(e).$ So, let us define $U=\{x\in X(e):u\preceq x\}$ and
notice that $\left\vert U\right\vert \leq5.$ Also, since $e\not \in R,$ all
edges of $R$ precede it and are not counted in $\nu(e).$ We conclude that
$\nu(e)\leq\sum_{x\in U}(\Delta-\deg_{R}x)-2,$ where we subtract $2$ since $e$
is counted twice in the sum and it should not be counted at all. Notice that
$\sum_{x\in U}\deg_{R}x\geq2(\left\vert U\right\vert -2).$ So, for $\left\vert
U\right\vert \leq4$ we obtain $\nu(e)\leq\left\vert U\right\vert
\Delta-2(\left\vert U\right\vert -2)-2\leq4\Delta-6\leq5\Delta-9$ for
$\Delta\geq3.$

It remains to consider the case $\left\vert U\right\vert =5.$ Since $R$ is a
$P_{\geq3}$-system, at least two vertices $x$ of $U$ must have $\deg_{R}%
x\geq2,$ while the remaining vertices have $\deg_{R}x\geq1.$ This implies
$\sum_{x\in U}\deg_{R}x\geq7,$ so $\nu(e)\leq5\Delta-7-2=5\Delta-9.$
\end{proof}

The assumption $\Delta\geq4$ in Lemma \ref{Lemma_linear} cannot be dropped,
since for a suitable $P_{\geq3}$-factor of the Petersen graph $\mathcal{G}$ is
not $6$-degenerate. Nevertheless, for $\Delta=3$ the desired bound
$\chi_{\mathrm{rn}}^{\prime}(G)\leq5\Delta-8=7$ is established in
\cite{PetrusevskiSkrekovski}.

\begin{proposition}
\label{Prop_sf}Let $G$ be a loopless multigraph with $\Delta=\Delta(G)\geq3$
which has a dominating $P_{\geq3}$-system. Then $\chi_{\mathrm{rn}}^{\prime
}(G)\leq5\Delta-8$.
\end{proposition}

\begin{proof}
For $\Delta=3$ the bound to prove is $\chi_{\mathrm{rn}}^{\prime}(G)\leq7,$
and this holds for every subcubic multigraph \cite{PetrusevskiSkrekovski}. So
let $\Delta\geq4$ and let $R$ be a dominating $P_{\geq3}$-system of $G$. By
Lemma \ref{Lemma_linear} the graph $\mathcal{G}$ is $\left(  5\Delta-9\right)
$-degenerate, so it has a proper vertex-coloring $\varphi$ with at most
$5\Delta-8$ colors. By Lemma \ref{Lemma_conflict}, $\varphi$ is a proper
edge-coloring of $G$ in which every edge of $R$ is rich. Every component of
$R$ is a path with at least two edges, and $V(G)\setminus V(R)$ is
independent, so Lemma \ref{Lemma_service} implies that $\varphi$ is a
rich-neighbor coloring.
\end{proof}

It remains to characterize the multigraphs which have a dominating $P_{\geq3}%
$-system. We need several notions. A graph $F$ is \emph{factor-critical} if
$F-v$ has a perfect matching for every $v\in V(F)$. A \emph{sun} is a graph
obtained from a factor-critical graph $F$ on at least three vertices by
attaching a new pendant vertex $x^{\prime}$ to every vertex $x$ of $F$. The
graphs $K_{1}$ and $K_{2}$ are also regarded as suns. We call $F$ the
\emph{core} of the sun, and the pairs $xx^{\prime}$ its \emph{pendant pairs}.
The pendant pair of $K_{2}$ consists of its two vertices. For a graph $G$ we
denote by $\mathrm{sun}(G)$ the number of components of $G$ which are suns.
For a multigraph $G$ we denote by $k_{2}(G)$ the number of components of $G$
on exactly two vertices. Finally, a $P_{4}$-factor of $G$ is a set of
vertex-disjoint paths on four vertices which cover $V(G)$. A $\{P_{4},P_{5}%
\}$-factor is defined in the same way, with paths on four or five vertices.

Kaneko \cite{Kaneko} proved that a graph $G$ has a $P_{\geq3}$-factor if and
only if
\begin{equation}
\mathrm{sun}(G-T)\leq2\left\vert T\right\vert \qquad\text{for every
}T\subseteq V(G). \label{For_kaneko}%
\end{equation}
Our characterization has the same form, with the suns replaced by the
components on two vertices.

The next lemma, whose proof uses two constructions of Kano, Katona and
Kir\'{a}ly \cite{KanoKatonaKiraly}, shows how a sun can be covered by paths
after one or two of its vertices are removed.

\begin{lemma}
\label{Lemma_sun}Let $D$ be a sun distinct from $K_{1}$. Then the following holds.

\begin{itemize}
\item[(a)] If $vv^{\prime}$ is a pendant pair of $D,$ then $D-\{v,v^{\prime
}\}$ has a $P_{4}$-factor.

\item[(b)] If $\left\vert V(D)\right\vert >2$ and $v^{\prime}$ is a pendant
vertex of $D,$ then $D-v^{\prime}$ has a $\{P_{4},P_{5}\}$-factor.
\end{itemize}
\end{lemma}

\begin{proof}
For $D=K_{2}$ the assertion (a) holds with the empty factor, so let $D$ have a
core $F$ on at least three vertices, and let $x^{\prime}$ denote the pendant
vertex attached to $x\in V(F)$.

$(a)$ Let $M$ be a perfect matching of $F-v,$ which exists because $F$ is
factor-critical. Replacing every edge $xy$ of $M$ by the path $x^{\prime
}xyy^{\prime}$ we obtain vertex-disjoint paths on four vertices which cover
$V(D)\setminus\{v,v^{\prime}\}$.$\bigskip$

$(b)$ Let $v$ be the vertex of $F$ to which $v^{\prime}$ is attached, let $x$
be a neighbor of $v$ in $F$ and let $M$ be a perfect matching of $F-x$. Then
$M$ covers $v,$ say $vy\in M$. We take the path $y^{\prime}yvxx^{\prime}$ on
five vertices, and we replace every edge $ab$ of $M$ other than $vy$ by the
path $a^{\prime}abb^{\prime}$. These paths are vertex-disjoint and they cover
$V(D)\setminus\{v^{\prime}\}$.
\end{proof}

Let us now characterize the multigraphs with a dominating $P_{\geq3}$-system.

\begin{theorem}
\label{Tm_char}A loopless multigraph $G$ has a dominating $P_{\geq3}$-system
if and only if
\begin{equation}
k_{2}(G-T)\leq2\left\vert T\right\vert \qquad\text{for every }T\subseteq V(G).
\label{For_char}%
\end{equation}

\end{theorem}

\begin{proof}
Assume first that $G$ has a dominating $P_{\geq3}$-system $R.$ Let $T\subseteq
V(G)$ and let $C$ be a component of $G-T$ on two vertices. The two vertices of
$C$ are adjacent, since $C$ is connected. The set $V(G)\setminus V(R)$ is
independent, so at least one vertex $x$ of $C$ belongs to $V(R).$ Let $P$ be
the path of $R$ which contains $x.$ Since $P$ has at least three vertices, it
is not contained in $C.$ So, $P$ contains an edge of $R$ which connects a
vertex of $C$ and a vertex of $T.$ Distinct components of $G-T$ give distinct
such edges, which implies%
\[
k_{2}(G-T)\leq\sum_{v\in T}\deg_{R}(v)\leq2\left\vert T\right\vert ,
\]
where the second inequality holds since $\deg_{R}(v)\leq2$ for every
$v.\smallskip$

Assume now that (\ref{For_char}) holds. We will use Kaneko's theorem, which is
established only for graphs, so let us first reduce the problem to graphs. A
dominating $P_{\geq3}$-system of a multigraph $G$ contains no parallel edges.
So, it is also a dominating $P_{\geq3}$-system of the underlying simple graph
of $G,$ and vice versa. Also, $k_{2}(G-T)$ has the same value in $G$ and in
its underlying simple graph, so (\ref{For_char}) holds for $G$ if and only if
it holds for its underlying simple graph. Hence, it suffices to prove the
theorem for graphs.\begin{figure}[h]
\begin{center}
\includegraphics[scale=1.05]{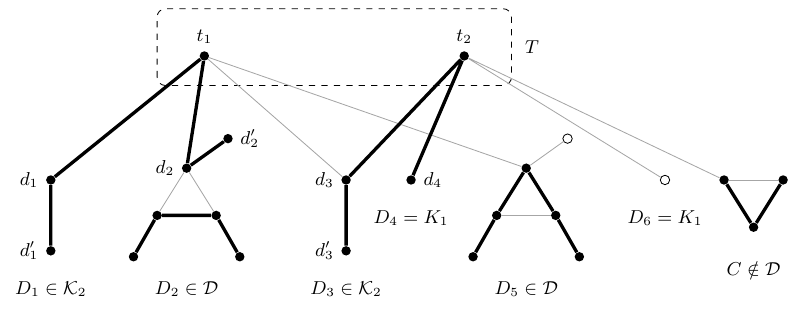}
\end{center}
\caption{An illustration of the construction in the proof of Theorem
\ref{Tm_char} with $T=\{t_{1},t_{2}\}$. Bold edges belong to $R$. The two
copies of $t_{1}$ are matched to the suns $D_{1}$ and $D_{2},$ which gives the
path $d_{1}^{\prime}d_{1}t_{1}d_{2}d_{2}^{\prime},$ and the rest of $D_{2}$ is
covered by a $P_{4}$-factor by Lemma \ref{Lemma_sun}(a). The two copies of
$t_{2}$ are matched to $D_{3}$ and $D_{4}=K_{1},$ which gives the path
$d_{3}^{\prime}d_{3}t_{2}d_{4}$. In the unmatched sun $D_{5}$ one pendant
vertex is left uncovered and the rest is covered by a path on five vertices by
Lemma \ref{Lemma_sun}(b), the unmatched sun $D_{6}=K_{1}$ is left uncovered,
and the component $C,$ which is not a sun, is covered by a $P_{\geq3}%
$-factor.}%
\label{Fig_char}%
\end{figure}

So, assume that $G$ is a graph. If (\ref{For_kaneko}) holds for $G,$ then $G$
has a $P_{\geq3}$-factor. A $P_{\geq3}$-factor is a dominating $P_{\geq3}%
$-system, so we are done. Otherwise, there exists $T\subseteq V(G)$ with
$f(T)=\mathrm{sun}(G-T)-2\left\vert T\right\vert >0,$ and we choose such $T$
so that $f(T)$ is maximum. We will construct a dominating $P_{\geq3}$-system
$R$ which covers every vertex of $T$ and leaves at most one vertex of each
component of $G-T$ uncovered. There are no edges between distinct components
of $G-T,$ so the uncovered vertices form an independent set, i.e. such $R$ is
dominating. The construction which follows is illustrated in Figure
\ref{Fig_char}.

Notice that a component of $G-T$ is either not a sun, or it is a sun, in which
case it is $K_{1},$ $K_{2},$ or a sun with a core on at least three vertices.
First, let us show that every component $C$ of $G-T$ which is not a sun has a
$P_{\geq3}$-factor. Let $X\subseteq V(C).$ Removing $X$ from $G-T$ affects
only the component $C,$ and $C$ is not a sun, so
\[
f(T\cup X)=\mathrm{sun}(G-T)+\mathrm{sun}(C-X)-2\left\vert T\right\vert
-2\left\vert X\right\vert .
\]
By the maximality of $T$ we have $f(T)-f(T\cup X)\geq0,$ i.e. $\mathrm{sun}%
(C-X)\leq2\left\vert X\right\vert .$ Since this holds for every $X\subseteq
V(C),$ the condition (\ref{For_kaneko}) holds for $C,$ so $C$ has a $P_{\geq
3}$-factor.

Next, we wish to assign two suns of $G-T$ to every $t\in T,$ so that no sun is
assigned twice and every component $K_{2}$ of $G-T$ is assigned to some $t.$
For that purpose, we define a bipartite graph $B$ as follows. One partite set
of $B$ is the set $\mathcal{T}$ which contains two copies of every $t\in T.$
The other partite set is the set $\mathcal{D}$ of all components of $G-T$
which are suns. A copy of $t$ is adjacent in $B$ to $D\in\mathcal{D}$ if and
only if $t$ has a neighbor in $D$ in the graph $G.$ Let $\mathcal{K}%
_{2}\subseteq\mathcal{D}$ be the set of all suns $D=K_{2}.$ The assignment we
aim at corresponds to a matching $M$ in $B$ which saturates both $\mathcal{T}$
and $\mathcal{K}_{2}.$ A necessary condition for the existence of such $M$ is
$\left\vert \mathcal{K}_{2}\right\vert \leq\left\vert \mathcal{T}\right\vert
\leq\left\vert \mathcal{D}\right\vert ,$ i.e. $k_{2}(G-T)\leq2\left\vert
T\right\vert \leq\mathrm{sun}(G-T).$ This condition holds, since the first
inequality is (\ref{For_char}), and the second one follows from $f(T)>0.$

In order to prove that such $M$ exists, we use Hall's theorem, which states
that a bipartite graph with partite sets $P$ and $Q$ has a matching which
saturates $P$ if and only if $\left\vert N(W)\right\vert \geq\left\vert
W\right\vert $ for every $W\subseteq P$. Let us first verify this condition
for $P=\mathcal{T}.$ Let $W\subseteq\mathcal{T},$ and let $X\subseteq T$ be
the set of vertices which have a copy in $W.$ Then $\left\vert W\right\vert
\leq2\left\vert X\right\vert .$ Every sun of $\mathcal{D}\setminus N_{B}(W)$
remains a sun in $G-(T\setminus X)$. This implies
\[
f(T\setminus X)=\mathrm{sun}(G-(T\setminus X))-2\left\vert T\setminus
X\right\vert \geq\mathrm{sun}(G-T)-\left\vert N_{B}(W)\right\vert -2\left\vert
T\right\vert +2\left\vert X\right\vert .
\]
By the maximality of $T$ we have $f(T)-f(T\setminus X)\geq0,$ so $\left\vert
N_{B}(W)\right\vert \geq2\left\vert X\right\vert \geq\left\vert W\right\vert
.$ Hence, $B$ has a matching which saturates $\mathcal{T}.$

Let us now verify Hall's condition for $P=\mathcal{K}_{2}.$ Let $W\subseteq
\mathcal{K}_{2},$ and let $X\subseteq T$ be the set of vertices which have a
neighbor in some sun of $W.$ Then $N_{B}(W)$ consists of both copies of
the vertices of $X,$ so $\left\vert N_{B}(W)\right\vert =2\left\vert
X\right\vert .$ Every sun of $W$ is a component of $G-X,$ so (\ref{For_char})
applied to $X$ implies
\[
\left\vert W\right\vert \leq k_{2}(G-X)\leq2\left\vert X\right\vert
=\left\vert N_{B}(W)\right\vert .
\]
Hence, $B$ has a matching which saturates $\mathcal{K}_{2}.$

It remains to establish that there is a matching $M$ which saturates
$\mathcal{T}$ and $\mathcal{K}_{2}$ simultaneously. For that purpose, take a
matching of $B$ which saturates $\mathcal{K}_{2},$ and augment it until it is
maximum. Augmenting keeps every matched vertex matched, so the obtained
maximum matching $M$ still saturates $\mathcal{K}_{2}.$ Since some matching
saturates $\mathcal{T},$ the matching $M$ has $\left\vert \mathcal{T}%
\right\vert $ edges, so it saturates $\mathcal{T}$ as well.

We are now in a position to construct the dominating $P_{\geq3}$-system $R.$
Let $t\in T,$ and let $D_{1}$ and $D_{2}$ be the suns matched by $M$ to the
two copies of $t.$ For $j=1,2,$ let $d_{j}$ be a neighbor of $t$ in $D_{j},$
and if $D_{j}\not =K_{1},$ let $d_{j}^{\prime}$ be the vertex which forms a
pendant pair with $d_{j}.$ We take the path $d_{1}^{\prime}d_{1}td_{2}%
d_{2}^{\prime},$ where $d_{j}^{\prime}$ is omitted if $D_{j}=K_{1}.$ This path
has three, four or five vertices. The remaining vertices of $D_{1}$ and
$D_{2}$ are covered by a $P_{4}$-factor by Lemma \ref{Lemma_sun}(a). So, all
vertices of the suns matched by $M$ are covered.

Consider now a sun $D$ which is not matched by $M.$ Since $M$ saturates
$\mathcal{K}_{2},$ the sun $D$ is either $K_{1}$ or it has a core on at least
three vertices. If $D=K_{1},$ we leave its vertex uncovered. Otherwise, we
leave one pendant vertex of $D$ uncovered, and we cover the remaining vertices
of $D$ by a $\{P_{4},P_{5}\}$-factor by Lemma \ref{Lemma_sun}(b). Finally, we
cover every component of $G-T$ which is not a sun by a $P_{\geq3}$-factor.

Let $R$ be the union of the edge sets of all the obtained paths. Every sun is
matched by $M$ to at most one copy of a vertex of $T,$ so these paths are
vertex-disjoint. Each of them has at least three vertices, so $R$ is a
$P_{\geq3}$-system. Every $t\in T$ is covered by $R,$ and every component of
$G-T$ contains at most one vertex not covered by $R.$ As noticed above, this
implies that $R$ is a dominating $P_{\geq3}$-system of $G.$
\end{proof}

\noindent Not every multigraph has a dominating $P_{\geq3}$-system. Let $G$ be
the spider obtained from $K_{1,3}$ by subdividing every edge once, and let $w$
be its center. Then $G-w$ consists of three components on two vertices, so
(\ref{For_char}) fails for $T=\{w\}$. The same holds whenever a multigraph has
a vertex $w$ and three paths $w\,a_{i}\,b_{i}$ with $d(a_{i})=2$ and
$d(b_{i})=1$. Since the rest of such a multigraph is arbitrary, examples exist
for every $\Delta\geq3$.

Next we show that the condition (\ref{For_char}) holds whenever the maximum
degree of $G$ is not much larger than its minimum degree.

\begin{corollary}
\label{Cor_degree}Let $G$ be a graph with $\delta=\delta(G)\geq2$ and
$\Delta=\Delta(G)$. If $\Delta\leq4\delta-4,$ then $G$ has a dominating
$P_{\geq3}$-system.
\end{corollary}

\begin{proof}
Let $T\subseteq V(G),$ and let $U$ be the set of the vertices of the
components of $G-T$ on two vertices. Then $\left\vert U\right\vert
=2k_{2}(G-T)$. A vertex $u\in U$ has exactly one neighbor outside $T,$ namely
the other vertex of its component. Since $d(u)\geq\delta,$ at least $\delta-1$
edges join $u$ to $T$. On the other hand, at most $\Delta\left\vert
T\right\vert $ edges are incident to vertices of $T$. Hence
\[
2\left(  \delta-1\right)  k_{2}(G-T)\leq\Delta\left\vert T\right\vert .
\]
Since $\delta\geq2$ and $\Delta\leq4\delta-4,$ this gives
\[
k_{2}(G-T)\leq\frac{\Delta}{2\left(  \delta-1\right)  }\left\vert T\right\vert
\leq2\left\vert T\right\vert .
\]
So (\ref{For_char}) holds, and Theorem \ref{Tm_char} gives the assertion.
\end{proof}

\noindent Proposition \ref{Prop_sf} now further implies that every graph $G$
with $3\leq\Delta\leq4\delta-4$ satisfies the bound $\chi_{\mathrm{rn}%
}^{\prime}(G)\leq5\Delta-8$. The regular case is now a simple consequence.

\begin{corollary}
\label{Cor_reg}Every $\Delta$-regular graph $G$ with $\Delta\geq3$ satisfies
\[
\chi_{\mathrm{rn}}^{\prime}(G)\leq5\Delta-8.
\]

\end{corollary}

\begin{proof}
A $\Delta$-regular graph satisfies $\delta(G)=\Delta\geq2$. Moreover,
$\Delta\leq4\Delta-4$ holds for every $\Delta\geq2$. So $G$ has a dominating
$P_{\geq3}$-system by Corollary \ref{Cor_degree}, and Proposition
\ref{Prop_sf} gives the bound.
\end{proof}

\noindent The condition $\Delta\leq4\delta-4$ in Corollary \ref{Cor_degree}
cannot be relaxed to $\Delta\leq4\delta-3$. Indeed, let $G$ consist of five
internally disjoint paths $a\,x_{i}\,y_{i}\,b,$ $1\leq i\leq5$. Then
$\delta(G)=2$ and $\Delta(G)=5=4\delta(G)-3$. Removing $T=\{a,b\}$ leaves the
five components $x_{i}y_{i}$ on two vertices, so $k_{2}(G-T)=5>4=2\left\vert
T\right\vert $. Hence, by Theorem \ref{Tm_char}, $G$ has no dominating
$P_{\geq3}$-system.

\section{Concluding remarks}

We have shown that the rich-neighbor chromatic index is bounded by a linear
function of the maximum degree. Every loopless multigraph $G$ with $\Delta
\geq3$ satisfies $\chi_{\mathrm{rn}}^{\prime}(G)\leq6\Delta-8,$ and every
$\Delta$-regular graph with $\Delta\geq3$ satisfies $\chi_{\mathrm{rn}%
}^{\prime}(G)\leq5\Delta-8$. The multigraphs with $\Delta\leq2$ need at most
four colors. It remains open how far these bounds can be pushed down by
refining the approach presented here or by entirely different techniques. Our
study of the problem leads us to believe that the tight upper bound for graphs
is considerably lower, as stated in the conjecture below.

Let us call a connected graph \emph{exceptional} if it is one of the five
graphs $C_{4},$ $C_{7},$ $K_{4},$ $K_{2,2,2}$ and $K_{6}$. Each of them is
regular and satisfies $\chi_{\mathrm{rn}}^{\prime}(G)=2\Delta,$ which exceeds
the lower bound of Observation \ref{Prop_lower} by one. We know of no further
graph which needs more than $2\Delta-1$ colors, which leads us to the
following conjecture.

\begin{conjecture}
\label{Con_exceptional}Every connected graph $G$ with $\Delta=\Delta(G)\geq2$
which is not exceptional satisfies
\[
\chi_{\mathrm{rn}}^{\prime}(G)\leq2\Delta-1.
\]

\end{conjecture}

By Observation \ref{Prop_lower}, the bound of Conjecture \ref{Con_exceptional}
would be best possible. For $\Delta=2$ the conjecture holds by Proposition
\ref{Prop_small}. For $\Delta=3$ it states that every connected graph with
$\Delta=3$ other than $K_{4}$ satisfies $\chi_{\mathrm{rn}}^{\prime}(G)\leq5,$
which is Conjecture 1.2 of \cite{PetrusevskiSkrekovski} restricted to graphs,
as there it is stated for connected subcubic loopless multigraphs. Thus
Conjecture \ref{Con_exceptional} extends that conjecture to an arbitrary
maximum degree.

\bigskip

\bigskip\noindent\textbf{Acknowledgments.}~~The authors acknowledge the
partial support by Keio University SPRING scholarship JPMJSP2123 and JST ERATO
JPMJER2301, by the Slovenian Research and Innovation Agency program P1--0383,
project J1--3002 and the annual work program of Rudolfovo, by Project
KK.01.1.1.02.0027 co-financed by the European Regional Development Fund, by
the NextGeneration EU foundation via IP--UNIST--17 (GEORAZ), and by the
bilateral Slovenian--Croatian project BI--HR/25--27--004.

\vspace{2mm}\bigskip\noindent\textbf{AI declaration.}~~Anthropic Claude.ai was
used in the preparation of this manuscript, for assistance with the arguments
and their exposition, and for verification of the proofs. All the results were
verified by the authors, who take full responsibility for the content.

\end{document}